%% file: Kida_v5.tex
\documentclass[12pt]{amsart}

\input{mystyle}

\newcommand{\otimesL}{\otimes^{\mathsf{L}}}

\newcommand{\DP}{D^{\perf}}
\newcommand{\DPT}{D^{\perf}_{\tors}}
\newcommand{\DPTM}{D^{\perf}_{\mu=0}}

\DeclareMathOperator{\RG}{\mathsf{R}\Gamma}
\DeclareMathOperator{\Iw}{Iw}

\DeclareMathOperator{\perf}{perf}
\DeclareMathOperator{\tors}{tors}
\DeclareMathOperator{\Sel}{Sel}
\DeclareMathOperator{\Tor}{Tor}

\title[Kida's formula]{Kida's formula via Selmer complexes}
\author[T.~Kataoka]{Takenori Kataoka}
\address{Department of Mathematics, Faculty of Science Division II, Tokyo University of Science.
1-3 Kagurazaka, Shinjuku-ku, Tokyo 162-8601, Japan}
\email{tkataoka@rs.tus.ac.jp}
\date{\today}

\begin{document}

\maketitle

\begin{abstract}
In Iwasawa theory, the $\lambda$, $\mu$-invariants of various arithmetic modules are fundamental invariants that measure the size of the modules.
Concerning the minus components of the unramified Iwasawa modules, Kida proved a formula that describes the behavior of those invariants with respect to field extensions.
Subsequently, many analogues of Kida's formula have been found in various areas in Iwasawa theory.
In this paper, we present a novel approach to such analogues of Kida's formula, based on the perspective of Selmer complexes.
\end{abstract}

\section{Introduction}\label{sec:intro}

Let $p$ be an odd prime number and $k_{\infty}/k$ a $\Z_p$-extension of a number field $k$.
One of the fundamental objects of study in Iwasawa theory is the unramified Iwasawa module $X(k_{\infty})$ for $k_{\infty}$, which is defined as the Galois group of the maximal unramified abelian pro-$p$ extension over $k_{\infty}$.
There are also many variants of Iwasawa modules defined by changing the ramification condition.
It is known that such Iwasawa modules are endowed with a (finitely generated) module structure over the Iwasawa algebra $\Lambda = \Z_p[[\Gal(k_{\infty}/k)]]$.
Furthermore, the scope of Iwasawa theory has been extended to other arithmetic situations, such as Selmer groups of elliptic curves and, more generally, of $p$-adic representations.

For a finitely generated $\Lambda$-module $M$, we have the $\lambda$, $\mu$-invariants $\lambda(M)$, $\mu(M)$.
They are defined via the structure theorem of modules over $\Lambda$ (we set $\lambda(M) = \mu(M) = + \infty$ if $M$ is not torsion).
The $\lambda$, $\mu$-invariants of arithmetic modules play important roles in Iwasawa theory.
For instance, the so-called Iwasawa class number formula describes the growth of the $p$-parts of the class numbers of intermediate number fields of $k_{\infty}/k$ by using those invariants associated to the unramified Iwasawa module $X(k_{\infty})$.

Kida \cite{Kid80} proved a formula that describes the behavior of those invariants when the $\Z_p$-extension varies.
To be concrete, let $K/k$ be a finite $p$-extension of CM-fields and consider the minus components $X(K_{\infty})^-$ and $X(k_{\infty})^-$ of the unramified Iwasawa modules for the cyclotomic $\Z_p$-extensions $K_{\infty}$ and $k_{\infty}$.
Then Kida's formula is of the form
\[
 \lambda(X(K_{\infty})^-) - \delta
 = [K_{\infty}: k_{\infty}] (\lambda(X(k_{\infty})^-) - \delta) + \sum_{w^+ \nmid p} (e_{w^+}(K_{\infty}^+/k_{\infty}^+) - 1)
\]
(see Theorem \ref{thm:Kida} for the precise statement).
This is regarded as an analogue of the classical Riemann--Hurwitz formula for coverings of Riemann surfaces.

Then, in turn, many analogues of Kida's formula have been found in various areas in Iwasawa theory.
For example, Hachimori--Matsuno \cite{HM99} obtained an analogue for Selmer groups of elliptic curves that are ordinary at $p$-adic primes.
Moreover, Pollack--Weston \cite{PW06} generalized that to $p$-adic representations.
See \S \ref{ss:compar} for more information on the previous research.

\subsection{Main results}\label{ss:main_results}

In this paper, we provide a powerful method to derive analogues of Kida's formula.
In a nutshell, the method is based on a key algebraic theorem about $\lambda$-invariants of perfect complexes, by applying which to arithmetic complexes we can obtain analogues of Kida's formula.
Let us explain the main results without introducing detailed notations.

\subsubsection{The algebraic theorem}\label{sss:alg}

Firstly we explain the key algebraic theorem.
See \S \ref{sec:alg} for the precise statement and its proof.

Let $\Lambda$ be a ring that is isomorphic to the power series ring $\Z_p[[T]]$ and we will deal with perfect complexes over $\Lambda$.
It is a key idea of this paper to introduce the following notions:
We say $\mu = 0$ for a perfect complex to mean that all the cohomology groups satisfy $\mu = 0$.
Also, we define the $\lambda$-invariant of a perfect complex as the alternating sum of the $\lambda$-invariants of the cohomology groups.

\begin{thm}[Theorem \ref{thm:Kida_alg}]\label{thm:Kida_alg1}
Let $G$ be a finite $p$-group.
Let $C$ be a perfect complex over the group ring $\Lambda[G]$, and put $\ol{C} = \Lambda \otimesL_{\Lambda[G]} C$.
Then $\mu = 0$ for $C$ is equivalent to $\mu = 0$ for $\ol{C}$.
If these equivalent conditions hold, then we have $\lambda(C) = (\# G) \lambda(\ol{C})$.
\end{thm}

\subsubsection{Kida's formula for ordinary representations}\label{sss:app_ord}

In \S \ref{sec:pf_thm}, to illustrate how useful Theorem \ref{thm:Kida_alg1} is, we apply it to $p$-adic representations that are ordinary in some sense.

Let $T$ be a $p$-adic representation over $\Z_p$ of the absolute Galois group of a number field $k$.
We suppose that $T$ is ordinary; this roughly means that, for each finite prime $v$ of $k$, we are given a suitable submodule $T_v^+$ of $T$ that is stable under the action of the absolute Galois group of $k_v$.
Using this data $\cT^+$, we define the Selmer group $\Sel_{\cT^+}(A/k_{\infty})$ for a $\Z_p$-extension $k_{\infty}/k$.

Let $K_{\infty}/k_{\infty}$ a finite $p$-extension.
By applying Theorem \ref{thm:Kida_alg1} to the relevant Selmer complex, we will obtain the following (the details will be given in \S \ref{sec:pf_thm}).

\begin{thm}[Theorem \ref{thm:main_ord}]\label{thm:main_ord1}
Under certain standard hypotheses, $\mu(\Sel_{\cT^+}(A/K_{\infty})^{\vee}) = 0$ is equivalent to $\mu(\Sel_{\cT^+}(A/k_{\infty})^{\vee}) = 0$.
If these equivalent conditions hold, then we have an explicit relation between $\lambda(\Sel_{\cT^+}(A/K_{\infty})^{\vee})$ and $\lambda(\Sel_{\cT^+}(A/k_{\infty})^{\vee})$.
\end{thm}

Note that Theorem \ref{thm:main_ord1} has an overlap with Pollack--Weston \cite{PW06}.
However, our method is different, and indeed we have succeeded in removing several simplifying assumptions.
For instance, the formulation of \cite{PW06} cannot be applied to $T = \Z_p(1)$ to recover results for classical Iwasawa modules, while we can as will be explained shortly.

\subsubsection{Arithmetic applications}\label{sss:app_arith}

We will see in \S \ref{sec:arith} that our method is powerful enough to obtain various analogues of Kida's formula.
The results in \S\S \ref{ss:Iw}--\ref{ss:ell} are deduced from Theorem \ref{thm:main_ord1}, while we have to use Theorem \ref{thm:Kida_alg1} again to deduce the one in \S \ref{ss:ell_ss}.
Here we summarize the objects of study in each subsection.

\begin{itemize}
\item[\S \ref{ss:Iw}:]
The $p$-ramified Iwasawa modules for the cyclotomic $\Z_p$-extensions of totally real fields.
The result is essentially equivalent to the original Kida's formula by the Kummer duality, as will be discussed in \S \ref{ss:Iw2}.
\item[\S \ref{ss:Iw2}:]
The minus components of split Iwasawa modules for the cyclotomic $\Z_p$-extensions of CM-fields.
This is a variation of the original Kida's formula, and indeed they are equivalent to each other.
\item[\S \ref{ss:Iw4}:]
The $\Sigma$-ramified Iwasawa modules for ordinary CM-fields, where $\Sigma$ is a $p$-adic CM-type.
By specializing to the case where the base field is an imaginary quadratic field, we recover a recent result of Murakami \cite{Mur}.
The general case seems to be new.
\item[\S \ref{ss:ell}:]
The Selmer groups for ordinary elliptic curves.
This recovers the result of Hachimori--Matsuno \cite{HM99}.
\item[\S \ref{ss:ell_ss}:]
The (multiply-)signed Selmer groups for supersingular elliptic curves.
This basically recovers a result of Lim \cite{Lim22}.
\end{itemize}

\subsubsection{Analogues for $p$-adic Lie extensions}\label{sss:Lie}

Finally in \S \ref{sec:Lie}, we discuss Kida's formula in the context of non-commutative Iwasawa theory, developed by Coates--Fukaya--Kato--Sujatha--Venjakob \cite{CFKSV05}, among others.
More precisely, we consider the case where $\Gal(K_{\infty}/k_{\infty})$ is a non-trivial, pro-$p$, $p$-adic Lie group without $p$-torsion (so in particular $K_{\infty}/k_{\infty}$ is infinite).
We will show that our method established in \S \S \ref{sec:alg}--\ref{sec:arith} can be also applied to this setting, which results in a formula on the $\Z_p[[\Gal(K_{\infty}/k_{\infty})]]$-rank of the Selmer group.
Our result recovers a theorem of Hachimori--Sharifi \cite{HS05} concerning the minus components of unramified Iwasawa modules for CM-fields.

\subsection{Review of previous work}\label{ss:compar}

There are a number of methods to deduce analogues of Kida's formula.
The original proof by Kida \cite{Kid80} is based on genus theory and studies the behavior of ideals with respect to field extensions.
Then Iwasawa \cite{Iwa81} gave an alternative proof via representation theory.
Sinnott \cite{Sin84} gave the third proof from analytic perspective, i.e., using $p$-adic $L$-functions.

The method of Iwasawa \cite[pages 285--288]{Iwa81} is adopted later by, e.g., Hachimori--Matsuno \cite{HM99} and Murakami \cite{Mur}.
Firstly we reduce the problem to the special case where the extension is cyclic of degree $p$.
For a cyclic group $G$ of order $p$, we have a classification theorem of $\Z_p[G]$-lattices.
This reduces the task to the computation of Herbrand quotients of arithmetic modules.

Pollack--Weston \cite{PW06} provided a different method.
Again after reducing to the case of cyclic extensions of degree $p$, they used the character decomposition with respect to the Galois action.
The characters are congruent modulo $p$, so to the twists one can apply a result of Weston \cite{Wes05} concerning congruent representations.

Greenberg \cite{Gree11} also obtained an interpretation of Kida's formula, mainly focusing on ordinary elliptic curves.
It is based on the idea that the non-primitive Selmer groups behave well for congruent representations (cf.~Greenberg--Vatsal \cite{GV00}) and then Kida's formula is obtained by comparing the non-primitive ones with the primitive ones.
This idea is adopted by Lim \cite{Lim22} to study the supersingular reduction case.
This method has an advantage that we can weaken the $\mu = 0$ hypothesis.
Hachimori--Sharifi \cite{HS05} also reproved Kida's formula with a weakened $\mu = 0$ hypothesis.
For a further attempt to remove the $\mu = 0$ hypothesis, see Lim \cite{Lim21}.

Our method in this paper is, at least superficially, different from all of these previous ones.
We deal with the general case directly without reducing to the cases of cyclic of order $p$.
We do not use the classification of $\Z_p[G]$-lattices or the congruence relation between twists of representations.

As for $p$-adic Lie extensions, our method is still different from Hachimori--Sharifi \cite{HS05} mentioned in \S \ref{sss:Lie}.
They determined the rank of the minus component of the unramified Iwasawa module by applying the original Kida's formula to the finite intermediate extensions and then investigating the asymptotic growth pattern of the $\lambda$-invariants.
Our method is more direct.

Finally, it is worth mentioning that Nichifor \cite[Theorem 3.3]{Nic04} proved the projectivity of the Iwasawa modules by applying Kida's formula.
Conceptually, this is in the opposite direction to our perspective in this paper; we know the perfectness of arithmetic objects in advance, from which Kida's formula follows.

\section*{Acknowledgments}

I am very grateful to Takamichi Sano for his encouragement and giving valuable comments.
This work is supported by JSPS KAKENHI Grant Number 22K13898.

\section{The key algebraic theorem}\label{sec:alg}

In this section, we prove the key algebraic theorem (Theorem \ref{thm:Kida_alg}) concerning perfect complexes.
The proof makes use of a module-theoretic version (Proposition \ref{prop:alg1}), which we first prove.

Let $\Lambda$ be a ring that is isomorphic to $\Z_p[[T]]$ with $p$ a prime number.
(The coefficient ring $\Z_p$ can be replaced by any $p$-adic integer ring, but to ease the notation we use $\Z_p$.)
Let $G$ be a finite $p$-group.
Throughout this paper we always mean left modules by simply saying modules.

\subsection{Module version}\label{ss:alg_mod}

For a finitely generated $\Lambda$-module $M$, we have the $\lambda$, $\mu$-invariants $\lambda(M), \mu(M) \geq 0$; we set $\lambda(M) = \mu(M) = + \infty$ if $M$ is not torsion.
Instead of reviewing their classical definition that depends on the structure theorem, let us observe their properties that we need in this paper:
\begin{itemize}
\item
We have $\mu(M) = 0$ if and only if $M$ is finitely generated over $\Z_p$.
\item
We have $\lambda(M) = \dim_{\Q_p}(\Q_p \otimes_{\Z_p} M)$.
\end{itemize}
In particular, if $\mu(M) = 0$, then $\lambda(M)$ coincides with the usual $\Z_p$-rank of $M$, denoted by $\rank_{\Z_p}(M)$.
Note that the property $\mu(M) = 0$ and the value $\lambda(M)$ are independent of the action of the indeterminate $T$.

Note also that, for a finitely generated $\Lambda$-module $M$, the property $\mu(M) = 0$ includes the assertion that $M$ is torsion.
In most of our main results (e.g., Theorem \ref{thm:main_ord1}), we do not know whether the arithmetic modules are torsion a priori, and the statements include consequences about that.

We will consider modules over the group ring $\Lambda[G]$.
The $\lambda$, $\mu$-invariants of finitely generated $\Lambda[G]$-modules are defined after forgetting the action of $G$.
For a $\Lambda[G]$-module $M$, define a $\Lambda$-module $\ol{M}$ by $\ol{M} = \Lambda \otimes_{\Lambda[G]} M$.

\begin{prop}\label{prop:alg1}
Let $M$ be a finitely generated $\Lambda[G]$-module.
\begin{itemize}
\item[(1)]
We have $\mu(M) = 0$ if and only if $\mu(\ol{M}) = 0$.
\item[(2)]
Suppose $\pd_{\Lambda[G]}(M) \leq 1$ ($\pd$ denotes the projective dimension).
If the equivalent conditions in (1) hold, then we have $\lambda(M) = (\# G) \lambda(\ol{M})$.
\end{itemize}
\end{prop}

\begin{proof}
This is a special case of \cite[Proposition 2.4]{Lim22}, which deals with more generally quasi-projective modules.
For the convenience of the readers, we reproduce the proof, following \cite[Proposition 8.5]{Kata_21}.

Since $\Z_p[G]$ is a local ring, claim (1) follows from topological Nakayama's lemma.
To show (2), suppose $\mu = 0$.
It is well-known that $\pd_{\Lambda}(M) \leq 1$ is equivalent to that $M$ does not contain a non-trivial finite submodule.
Therefore, $M$ is finitely generated and free over $\Z_p$.
Then $\pd_{\Z_p[G]}(M) \leq 1$ is equivalent to that $M$ is free over $\Z_p[G]$, since again $\Z_p[G]$ is local.
If we write $r$ for the rank of $M$ as a free $\Z_p[G]$-module, we have $\lambda(M) = (\# G) r$.
On the other hand, $\ol{M}$ is a free $\Z_p$-module of rank $r$.
These show $\lambda(M) = (\# G) \lambda(\ol{M})$, as claimed.
\end{proof}

\subsection{Complex version}\label{ss:alg_cpx}

We introduce the following notation.

\begin{itemize}
\item
Let $\DP(\Lambda)$ be the derived category of perfect complexes over $\Lambda$.
\item
Let $\DPT(\Lambda) \subset \DP(\Lambda)$ be the derived category of perfect complexes over $\Lambda$ whose cohomology groups are all torsion over $\Lambda$.
\item
Let $\DPTM(\Lambda) \subset \DP(\Lambda)$ be the derived category of perfect complexes over $\Lambda$ whose cohomology groups all satisfy $\mu = 0$.
\end{itemize}

For $C \in \DP(\Lambda)$, we say ``$\mu = 0$ for $C$'' if we have $C \in \DPTM(\Lambda)$.
By our convention on $\mu$, we have $\DPTM(\Lambda) \subset \DPT(\Lambda)$.

\begin{lem}\label{lem:Ser}
Let $C' \to C \to C''$ be a triangle in $\DP(\Lambda)$.
If two of $C$, $C'$, $C''$ are in $\DPT(\Lambda)$ (resp.~$\DPTM(\Lambda)$), then so is the remaining one.
\end{lem}

\begin{proof}
This follows immediately from the long exact sequence of the cohomology groups.
\end{proof}

\begin{defn}\label{defn:l_m}
For $C \in \DPT(\Lambda)$,
we define its $\lambda$-invariant by
\[
\lambda(C) = \sum_{i \in \Z} (-1)^i \lambda(H^i(C)).
\]
(Note that we may define $\mu(C)$ in the same way, but this is not useful for our purpose.)
\end{defn}

\begin{lem}\label{lem:alg2}
Let $C' \to C \to C''$ be a triangle in $\DPT(\Lambda)$.
Then we have $\lambda(C) = \lambda(C') + \lambda(C'')$.
\end{lem}

\begin{proof}
This follows from the long exact sequence  of the cohomology groups and the additivity of the $\lambda$-invariants with respect to exact sequences.
\end{proof}

We define $\DP(\Lambda[G])$, $\DPT(\Lambda[G])$, and $\DPTM(\Lambda[G])$ in the same way as for $\Lambda$.
Again, for $C \in \DP(\Lambda[G])$, the $\lambda$ and $\mu$ depend only on the information after forgetting the action of $G$.

For $C \in \DP(\Lambda[G])$, we put $\ol{C} = \Lambda \otimesL_{\Lambda[G]} C \in \DP(\Lambda)$.
The following is the key theorem.

\begin{thm}\label{thm:Kida_alg}
For $C \in \DP(\Lambda[G])$, we have $\mu = 0$ for $C$ if and only if $\mu = 0$ for $\ol{C}$.
If these equivalent conditions hold, then we have
$
\lambda(C) = (\# G) \lambda(\ol{C}).
$
\end{thm}

\begin{proof}
Choose a quasi-isomorphism
\[
C \simeq [C^a \to C^{a+1} \to \dots \to C^b],
\]
where $C^a, \dots, C^b$ are finitely generated projective (i.e., free) $\Lambda[G]$-modules.
Since $\ol{H^b(C)} \simeq H^b(\ol{C})$, by Proposition \ref{prop:alg1}(1) applied to $M = H^b(C)$, we have $\mu(H^b(C)) = 0$ if and only if $\mu(H^b(\ol{C})) = 0$.

We argue by induction on the length $b - a$.
Firstly suppose $b - a \leq 1$, so we may write $C \simeq [C^{b-1} \to C^b]$.
We may assume $\mu = 0$ for $C$ or $\ol{C}$, so in particular $C \in \DPT(\Lambda[G])$ or $\ol{C} \in \DPT(\Lambda)$.
In any case, the $\Lambda[G]$-ranks of $C^{b-1}$ and $C^b$ coincide.
Also, we have $\mu = 0$ for both $H^b(C)$ and $H^b(\ol{C})$, so in particular both are torsion.
Therefore, both $H^{b-1}(C)$ and $H^{b-1}(\ol{C})$ vanish and we obtain
\[
\lambda(C) = (-1)^b \lambda(H^b(C)),
\qquad
\lambda(\ol{C}) = (-1)^b \lambda(H^b(\ol{C})).
\]
Then the theorem follows from Proposition \ref{prop:alg1}(2), applied to $M = H^b(C)$.

Suppose $b - a \geq 2$.
We have both $\mu(H^b(C)) = 0$ and $\mu(H^b(\ol{C})) = 0$ as above, so we can take a nonzero annihilator $f \in \Lambda$ of $H^b(C)$ such that $\mu(\Lambda/(f)) = 0$.
Consider a commutative diagram
\[
\xymatrix{
	& & C^b \ar[r]^{f \times} \ar[d]_{*}
	& C^b \ar[d]^{=} \\
	C^a \ar[r]
	 & \cdots \ar[r]
	 & C^{b-1} \ar[r]
	 & C^b.
}
\]
Here, the map $*$ exists since $C^b$ is projective and $f$ annihilates $H^b(C)$.
We regard this diagram as a morphism from the complex $C' := [C^b \overset{f \times}{\to} C^b]$ to $C$, where $C'$ is located at degrees $b-1, b$.
Let $C''$ be its cone, so we have a triangle $C' \to C \to C''$.
By base-change, we also obtain a triangle $\ol{C'} \to \ol{C} \to \ol{C''}$.

By construction, $C''$ is located at degrees $a, a+1, \dots, b$ (since $b - a \geq 2$), and moreover we have $H^b(C'') = 0$.
Therefore, $C''$ can be indeed represented by a complex of finitely generated projective modules at degrees $a, a+1, \dots, b - 1$.
So we may apply the induction hypothesis to $C''$.

By the choice of $f$, we have $\mu = 0$ for $C'$ and $\ol{C'}$.
Then by Lemma \ref{lem:Ser}, we have the first and the third equivalences of
\[
\text{$\mu = 0$ for $C$}
 \Leftrightarrow \text{$\mu = 0$ for $C''$}
 \Leftrightarrow \text{$\mu = 0$ for $\ol{C''}$}
 \Leftrightarrow \text{$\mu = 0$ for $\ol{C}$}.
\]
The second equivalence holds by the induction hypothesis.
Thus we have proved the assertion concerning $\mu = 0$.

Suppose $\mu = 0$.
Then we have
\begin{align}
\lambda(C) 
& = \lambda(C') + \lambda(C'')
 = (\# G) \lambda(\ol{C'}) + (\# G) \lambda(\ol{C''})\\
& = (\# G) (\lambda(\ol{C'}) + \lambda(\ol{C''}))
 = (\# G) \lambda(\ol{C}).
\end{align}
Here, the first and the last equality follows from Lemma \ref{lem:alg2}.
The second equality follows from the induction hypothesis (and the $b - a = 1$ case).
This proves the assertion concerning the $\lambda$-invariants.
\end{proof}

\section{Application to ordinary representations}\label{sec:pf_thm}

In this section, we apply the established key algebraic theorem to the Selmer complex associated to ordinary representations.

\subsection{Statement of the theorem}\label{ss:state_thm}

We begin with the precise setup.
Let $p$ be an odd prime number.
Let $k$ be a number field.
Let $T$ be a free Galois representation over $\Z_p$ of the absolute Galois group $\Gal(\ol{k}/k)$ that is unramified at almost all primes.
Fundamental examples of $T$ are the Tate modules $\Z_p(1)$ (of the multiplicative group $\bG_m$) and $T_pE$ (of an elliptic curve $E$ over $k$).

Suppose that, for each finite prime $v$ of $k$, we are given a short exact sequence of free representations of $\Gal(\ol{k_v}/k_v)$
\[
0 \to T_v^+ \to T \to T_v^- \to 0
\]
such that we have $T_v^+ = T$ for almost all primes $v$.
In the applications we often take $T_v^+ = T$ for all non-$p$-adic primes $v$.
Let us write $\cT^+$ to mean this data $\{T_v^+\}_v$.

Put $A = T^{\vee}(1)$ and $A_v^{\pm} = (T_v^{\mp})^{\vee}(1)$, where $(-)^{\vee}$ denotes the Pontryagin dual and $(1)$ denotes the Tate twist.
Then we have an exact sequence
\[
0 \to A_v^+ \to A \to A_v^- \to 0
\]
for each finite prime $v$.

Let $S_p(k)$, $S_{\R}(k)$, and $S_{\C}(k)$ denote the set of $p$-adic, real, and complex places of $k$, respectively.
Put $S_{\infty}(k) = S_{\R}(k) \cup S_{\C}(k)$.

We impose the following assumption as a part of ordinarity.
As will be mentioned in \S \ref{ss:Sel_cpx}, it is equivalent to that the Euler characteristic of the associated Selmer complex is zero.

\begin{ass}\label{ass:Sel_rk}
We have
\[
\sum_{v \in S_p(k)} [k_v: \Q_p] \rank_{\Z_p}(T_v^-) = \sum_{v \in S_{\R}(k)} \rank_{\Z_p}(T^{j_v = -1}) + \sum_{v \in S_{\C}(k)} \rank_{\Z_p}(T).
\]
Here, for $v \in S_{\R}(k)$, we write $j_v$ for the complex conjugation at $v$.
\end{ass}

Next, we take a $\Z_p$-extension  $k_{\infty}/k$ satisfying the following.

\begin{ass}\label{ass:spl_fin}
Any finite primes split finitely in $k_{\infty}/k$ (i.e., do not split completely).
\end{ass}

This is satisfied if $k_{\infty}/k$ is the cyclotomic $\Z_p$-extension.

Let $K_{\infty}/k_{\infty}$ be a finite $p$-extension.
We define the Selmer group by
\[
\Sel_{\cT^+}(A/K_{\infty})
= \Ker \left(
H^1(K_{\infty}, A)
\to \bigoplus_{v} H^1(K_{\infty, v}, A_v^-)
\right),
\]
where $v$ runs over all finite primes of $k$ and the map is the natural one.
Here, in general we write
\[
H^1(K_{\infty, v}, -) = \bigoplus_{w \mid v} H^1(K_{\infty, w}, -)
\]
with $w$ running over the primes of $K_{\infty}$ lying above $v$.

We set
\[
\delta_{\cT^+}(T/K_{\infty}) 
= \lambda \left(\Ker \left( H^0(K_{\infty}, T) \to \bigoplus_v H^0(K_{\infty, v}, T_v^-) \right) \right),
\]
where $v$ runs over all finite primes of $k$.
Note that we often have $\delta_{\cT^+}(T/K_{\infty}) = 0$; for instance, this holds as soon as $H^0(K_{\infty}, T) = 0$ or there is a prime $v$ such that $T_v^- = T$ (i.e., $T_v^+ = 0$).
Indeed, in the applications in this paper, we have $\delta_{\cT^+}(T/K_{\infty}) = 0$, except for \S \ref{ss:Iw2}.

The main theorem in this section is the following.

\begin{thm}\label{thm:main_ord}
Suppose Assumptions \ref{ass:Sel_rk} and \ref{ass:spl_fin}.
We have $\mu(\Sel_{\cT^+}(A/K_{\infty})^{\vee}) = 0$ if and only if $\mu(\Sel_{\cT^+}(A/k_{\infty})^{\vee}) = 0$.
If these equivalent conditions hold, then we have
\begin{align}
& \lambda(\Sel_{\cT^+}(A/K_{\infty})^{\vee})
	- \lambda(H^0(K_{\infty}, A)^{\vee})
	- \delta_{\cT^+}(T/K_{\infty})\\
& \qquad = [K_{\infty}: k_{\infty}] \left[\lambda(\Sel_{\cT^+}(A/k_{\infty})^{\vee})
	- \lambda(H^0(k_{\infty}, A)^{\vee})
	- \delta_{\cT^+}(T/k_{\infty}) \right]\\
& \qquad \quad + \sum_{w} \left[ n_w(K_{\infty}/k_{\infty}) \lambda(H^0(k_{\infty, w}, A_v^-)^{\vee}) - \lambda(H^0(K_{\infty, w}, A_v^-)^{\vee}) \right],
\end{align}
where $w$ runs over the finite primes of $K_{\infty}$.
For each $w$, the $v$ in $A_v^-$ is the prime of $k$ lying below $w$, and $k_{\infty, w}$ denotes the completion at the prime below $w$.
Also, we put
\[
n_{w}(K_{\infty}/k_{\infty}) := e_{w}(K_{\infty}/k_{\infty}) f_{w}(K_{\infty}/k_{\infty})
\]
by using the ramification index $e_{w}(K_{\infty}/k_{\infty})$ and the inertia degree $f_{w}(K_{\infty}/k_{\infty})$.
\end{thm}

Note that $n_{w}(K_{\infty}/k_{\infty}) = e_{w}(K_{\infty}/k_{\infty})$ if $w \nmid p$, since by Assumption \ref{ass:spl_fin} the extension $k_{\infty}/k$ already involves the inert $\Z_p$-extensions for non-$p$-adic primes.

\subsection{The Selmer complex}\label{ss:Sel_cpx}

In this subsection, as a preliminary to the proof of Theorem \ref{thm:main_ord}, we introduce the Selmer complex and review its basic properties.

We keep the notation in \S \ref{ss:state_thm}.
Let us take a finite set $S$ of places of $k$ such that
\[
S \supset S_{\infty}(k) \cup S_p(k) \cup S_{\ram}(K_{\infty}/k) \cup S_{\ram}(T),
\]
where $S_{\ram}(K_{\infty}/k)$ and $S_{\ram}(T)$ denote the sets of ramified primes of $k$ for the extension $K_{\infty}/k$ and the representation $T$, respectively.
We moreover enlarge $S$ if necessary so that we have $T_v^+ = T$ for any $v \not \in S$.
We set $S_f = S \setminus S_{\infty}(k)$.

Let $k_S/k$ be the maximal extension of $k$ that is unramified outside $S$.
We consider the global and local Iwasawa cohomology complexes
\[
\RG_{\Iw}(k_S/K_{\infty}, -)
\text{ and }
\RG_{\Iw}(K_{\infty, v}, -) \quad (v \in S_f).
\]
For the basic properties of these complexes, see the author's paper \cite[\S 3]{Kata_12}, or Fukaya--Kato \cite[\S 1.6]{FK06} and Nekov\'{a}\v{r} \cite{Nek06} as more comprehensive references (\cite{FK06} deals with non-commutative cases).

The Selmer complex $C$ in the current situation (see \cite[\S 6.1]{Nek06}) is defined so that it satisfies a triangle
\begin{equation}\label{eq:Sel1}
C 
\to \RG_{\Iw}(k_S/K_{\infty}, T)
\to \bigoplus_{v \in S_f} \RG_{\Iw}(K_{\infty, v}, T_v^-),
\end{equation}
where the morphism between the global and local Iwasawa cohomology complexes is the localization.
Here we summarize properties of $C$.
Note that by changing $k$ if necessary (see the beginning of \S \ref{ss:pf_ord_thm}), we may suppose that $K_{\infty}/k$ is Galois and set $\cR_{K_{\infty}} = \Z_p[[\Gal(K_{\infty}/k)]]$.

\begin{itemize}
\item[(i)]
The complex $C$ is perfect over $\cR_{K_{\infty}}$.
More precisely, $C$ is quasi-isomorphic to a complex of the form $[C^0 \to C^1 \to C^2 \to C^3]$ with finitely generated projective $\cR_{K_{\infty}}$-modules $C^i$ located at degree $i$.
This follows from the fact that the global and local Iwasawa cohomology complexes are perfect (see \cite[Proposition 1.6.5(2)]{FK06}).

\item[(ii)]
Putting $\ol{C} = \cR_{k_{\infty}} \otimesL_{\cR_{K_{\infty}}} C$ with $\cR_{k_{\infty}} = \Z_p[[\Gal(k_{\infty}/k)]]$, we have a triangle
\[
\ol{C}
\to \RG_{\Iw}(k_S/k_{\infty}, T)
\to \bigoplus_{v \in S_f} \RG_{\Iw}(k_{\infty, v}, T_v^-)
\]
over $\cR_{k_{\infty}}$.
This follows from the corresponding compatibility about the global and local Iwasawa cohomology complexes (see \cite[Proposition 1.6.5(3)]{FK06}).

\item[(iii)]
The Euler characteristic of $C$ is zero.
This follows from the global and local Euler--Poincare characteristic formulas and Assumption \ref{ass:Sel_rk}.
Indeed, the Euler characteristic of $\RG_{\Iw}(k_S/K_{\infty}, T)$ (resp.~$\bigoplus_{v \in S_f} \RG_{\Iw}(K_{\infty, v}, T_v^-)$) is the right (resp.~left) hand side of the equation in Assumption \ref{ass:Sel_rk} (see \cite[Theorems (7.3.1) and (8.7.4)]{NSW08}).

\item[(iv)]
We have a triangle
\[
C
\to \bigoplus_{v \in S_f} \RG_{\Iw}(K_{\infty, v}, T_v^+)
\to \RG(k_S/K_{\infty}, A)^{\vee}[-2].
\]
This follows from the complex version of the global duality (see \cite[\S 5.4]{Nek06})
\[
\RG_{\Iw}(k_S/K_{\infty}, T)
\to \bigoplus_{v \in S_f} \RG_{\Iw}(K_{\infty, v}, T)
\to \RG(k_S/K_{\infty}, A)^{\vee}[-2].
\]
\end{itemize}

The cohomology groups of $C$ are described as follows.

\begin{prop}\label{prop:C_coh}
We have $H^i(C) = 0$ for $i \neq 1, 2, 3$
and exact sequences
\begin{equation}\label{eq:Sel2}
0 \to H^1(C) \to H^1_{\Iw}(k_S/K_{\infty}, T)
\to \bigoplus_{v \in S_f} H^1_{\Iw}(K_{\infty, v}, T_v^-)
\end{equation}
and
\begin{equation}\label{eq:Sel3}
0 
\to \Sel_{\cT^+}(A/K_{\infty})^{\vee} 
\to H^2(C) 
\to \bigoplus_{v \in S_f} H^0(K_{\infty, v}, A_v^-)^{\vee}
 \to H^0(K_{\infty}, A)^{\vee} 
\to H^3(C) \to 0.
\end{equation}
\end{prop}

\begin{proof}
By property (i), we have $H^i(C) = 0$ for $i \neq 0, 1, 2, 3$.
It is known that both $H^0_{\Iw}(k_S/K_{\infty}, T)$ and $H^0_{\Iw}(K_{\infty, v}, T_v^-)$ ($v \in S_f$) vanish by Assumption \ref{ass:spl_fin} (see Rubin \cite[Appendix B, Lemma 3.2]{Rub00}).
Then the definition of $C$ implies $H^0(C) = 0$ and \eqref{eq:Sel2}.
On the other hand, property (iv) and the local duality induce \eqref{eq:Sel3} since the Selmer group $\Sel_{\cT^+}(A/K_{\infty})$ coincides with the kernel of the natural homomorphism
\[
H^1(k_S/K_{\infty}, A)
\to \bigoplus_{v \in S_f} H^1(K_{\infty, v}, A_v^-),
\]
thanks to Assumption \ref{ass:spl_fin}.
\end{proof}

We note the following:
If there exists a finite prime $v$ of $K$ such that $T_v^- = T$ (i.e., $T_v^+ = 0$), then we have $H^1(C)_{\tors} = 0$, where $(-)_{\tors}$ denotes the maximal $\cR_{K_{\infty}}$-torsion submodule.
This easily follows by observing that we may take $C^0 = 0$ in property (i) in this case.
In general, we have the following description of $H^1(C)_{\tors}$.
The result should be known to experts (e.g., \cite[Proposition 8.11.5(ii), Corollary 9.1.7]{Nek06}), but we include a direct proof for convenience.

\begin{lem}\label{lem:main_pf1}
We have an exact sequence
\[
0 \to H^1(C)_{\tors} \to H^0(K_{\infty}, T)
\to \bigoplus_{v \in S_f} H^0(K_{\infty, v}, T_v^-).
\]
\end{lem}

\begin{proof}
Since taking $(-)_{\tors}$ is left exact, by \eqref{eq:Sel2}, it is enough to show functorial isomorphisms 
$H^1_{\Iw}(k_S/K_{\infty}, T)_{\tors} \simeq H^0(K_{\infty}, T)$ and
$H^1_{\Iw}(K_{\infty, v}, T_v^-)_{\tors} \simeq H^0(K_{\infty, v}, T_v^-)$.
Let us show only the former one as the latter one can be shown in a similar way.

We write $\otimes$ to mean $\otimes_{\Z_p}$.
Let us take an intermediate number field $K$ in $K_{\infty}/k$ such that $K_{\infty} = k_{\infty} K$.
Put $\Gamma = \Gal(K_{\infty}/K)$ and put $\Lambda = \Z_p[[\Gamma]]$.
Since $\Gamma$ is an open subgroup of $\Gal(K_{\infty}/k)$, being $\Lambda$-torsion is equivalent to being $\cR_{K_{\infty}}$-torsion.
Recall that, by Shapiro's lemma, we have $H^i_{\Iw}(k_S/K_{\infty}, T) \simeq H^i(k_S/K, \Lambda \otimes T)$ if we regard $\Lambda \otimes T$ as a $\Lambda$-module with $\Gal(k_S/K)$-action defined by
\[
\sigma \cdot (a \otimes x) = (a \sigma^{-1}) \otimes (\sigma x)
\]
for $\sigma \in \Gal(k_S/K)$, $a \in \Lambda$, and $x \in T$.

Let $f \in \Lambda$ be an arbitrary nonzero element.
By taking the $\Gal(k_S/K)$-cohomology of the tautological exact sequence
\[
0 \to \Lambda \otimes T \overset{(f \times) \otimes (\id)}{\to} \Lambda \otimes T \to (\Lambda/(f)) \otimes T \to 0,
\] 
we obtain the first isomorphism of
\begin{align}
H^1_{\Iw}(k_S/K_{\infty}, T)[f] 
& \simeq H^0(K,( \Lambda/(f)) \otimes T)\\
& \simeq H^0(K_{\infty}/K, (\Lambda/(f)) \otimes M),
\end{align}
where we put $M = H^0(K_{\infty},T)$.

We take a generator $\gamma$ of $\Gamma$.
By the snake lemma applied to the multiplication by $\gamma - 1$ to the tautological exact sequence
\[
0 \to \Lambda \otimes M \overset{(f \times) \otimes (\id)}{\to} \Lambda \otimes M \to (\Lambda/(f)) \otimes M \to 0,
\] 
we obtain an exact sequence
\[
0 \to (\Lambda \otimes M)^{\Gamma}/f \to ((\Lambda/(f)) \otimes M)^{\Gamma} 
\to (\Lambda \otimes M)_{\Gamma}[f] \to 0.
\]
We have $(\Lambda \otimes M)^{\Gamma} = 0$ and an isomorphism $(\Lambda \otimes M)_{\Gamma} \simeq M$ by sending $a \otimes x$ to $ax$ for $a \in \Lambda$ and $x \in M$.
Thus we obtain $((\Lambda/(f)) \otimes M)^{\Gamma} \simeq M[f]$.

We have so far obtained an isomorphism $H^1_{\Iw}(k_S/K_{\infty}, T)[f] \simeq M[f]$.
By taking the inductive limit with respect to $f$, we obtain $H^1_{\Iw}(k_S/K_{\infty}, T)_{\tors} \simeq M = H^0(K_{\infty},T)$, as claimed.
\end{proof}

\subsection{Proof of Theorem \ref{thm:main_ord}}\label{ss:pf_ord_thm}

Now we are ready to prove Theorem \ref{thm:main_ord}.

Put $G = \Gal(K_{\infty}/k_{\infty})$.
Take an intermediate number field $K$ in $K_{\infty}/k$ such that $K_{\infty} = k_{\infty} K$.
By replacing $k$ by $K \cap k_{\infty}$, we may identify $\Gal(K_{\infty}/k)$ with $\Gal(K_{\infty}/K) \times G$, so
\[
\Z_p[[\Gal(K_{\infty}/k)]] \simeq \Z_p[[\Gal(K_{\infty}/K)]][G].
\]
We will apply the result in \S \ref{sec:alg} to $\Lambda = \Z_p[[\Gal(K_{\infty}/K)]]$.

As for $\mu = 0$, by \eqref{eq:Sel3} and Assumption \ref{ass:spl_fin},
we have $\mu(H^3(C)) = 0$ and $\mu(\Sel_{\cT^+}(A/K_{\infty})^{\vee}) = \mu(H^2(C))$.
We proceed in the following manner.
\[
\xymatrix{
	\text{$\mu = 0$ for $C$} \ar@{<=>}[r]^{(\rm a)} \ar@{<=>}[d]_{(\rm b)}
	& \text{$\mu = 0$ for $\ol{C}$} \ar@{<=>}[d]^{(\rm b')} \\
	\mu(H^2(C)) = 0 \ar@{<=>}[d]_{(\rm c)}
	& \mu(H^2(\ol{C})) = 0 \ar@{<=>}[d]^{(\rm c')}\\
	\mu(\Sel_{\cT^+}(A/K_{\infty})^{\vee}) = 0
	& \mu(\Sel_{\cT^+}(A/k_{\infty})^{\vee}) = 0
}
\]
Claim (a) follows from Theorem \ref{thm:Kida_alg}.
Claim (c) is already shown.
Let us show claim (b).
One implication is trivial, so we suppose $\mu(H^2(C)) = 0$ and will deduce $\mu(H^i(C)) = 0$ for all $i$.
We have $H^i(C) = 0$ for $i \neq 1, 2, 3$ and $\mu(H^3(C)) = 0$ is already shown.
In particular, both $H^2(C)$ and $H^3(C)$ are torsion, so property (iii) on the Euler characteristic implies that $H^1(C)$ is also torsion.
Then Lemma \ref{lem:main_pf1} implies $\mu(H^1(C)) = 0$.
This shows claim (b).
Finally, claims (b') and (c') follow in the same manner as (b) and (c) respectively, thanks to property (ii).

Now we assume $\mu = 0$ and consider the $\lambda$-invariants.
By Proposition \ref{prop:C_coh} and Lemma \ref{lem:main_pf1}, we obtain
\begin{align}
\lambda(C) 
& = \lambda(H^2(C)) - \lambda(H^3(C)) - \lambda(H^1(C))\\
& = \lambda(\Sel_{\cT^+}(A/K_{\infty})^{\vee})
	- \lambda(H^0(K_{\infty}, A)^{\vee})
	+ \sum_{v \in S_f} \lambda(H^0(K_{\infty, v}, A_v^-)^{\vee})
	- \delta_{\cT^+}(T/K_{\infty}).
\end{align}
We have the corresponding formula for $\ol{C}$ instead of $C$.
By Theorem \ref{thm:Kida_alg}, we have $\lambda(C) = [K_{\infty}: k_{\infty}] \lambda(\ol{C})$.
Combining these, we obtain
\begin{align}
& \lambda(\Sel_{\cT^+}(A/K_{\infty})^{\vee})
	- \lambda(H^0(K_{\infty}, A)^{\vee})
	- \delta_{\cT^+}(T/K_{\infty})\\
& = [K_{\infty}: k_{\infty}] \left[ \lambda(\Sel_{\cT^+}(A/k_{\infty})^{\vee})
	- \lambda(H^0(k_{\infty}, A)^{\vee})
	- \delta_{\cT^+}(T/k_{\infty}) \right]\\
& \quad + \sum_{v \in S_f} \left[ [K_{\infty}: k_{\infty}] \lambda(H^0(k_{\infty, v}, A_v^-)^{\vee}) - \lambda(H^0(K_{\infty, v}, A_v^-)^{\vee}) \right].
\end{align}
For each finite prime $v$ of $k$, it is easy to see
\begin{align}
& [K_{\infty}: k_{\infty}] \lambda(H^0(k_{\infty, v}, A_v^-)^{\vee}) - \lambda(H^0(K_{\infty, v}, A_v^-)^{\vee})\\
& \quad = \sum_{w \mid v} \left( n_w(K_{\infty}/k_{\infty}) \lambda(H^0(k_{\infty, w}, A_v^-)^{\vee}) - \lambda(H^0(K_{\infty, w}, A_v^-)^{\vee}) \right),
\end{align}
where $w$ runs over the primes of $K_{\infty}$ that lie above $v$.
This quantity vanishes for $v \not \in S_f$ since such a prime splits completely in $K_{\infty}/k_{\infty}$.
This completes the proof of Theorem \ref{thm:main_ord}.

\section{Arithmetic applications}\label{sec:arith}

In this section, we use Theorem \ref{thm:main_ord} to deduce concrete applications.
In \S \ref{ss:ell_ss}, we also deal with a non-ordinary case by applying Theorem \ref{thm:Kida_alg}.

\subsection{Totally real fields}\label{ss:Iw}

Let $K_{\infty}/k_{\infty}$ be a finite $p$-extension between the cyclotomic $\Z_p$-extensions of totally real fields.
Let $X_p(K_{\infty})$ denote the $p$-ramified Iwasawa module for $K_{\infty}$ (that is, ramification at $p$-adic primes is allowed).
The result is the following.
It is already known by Iwasawa \cite{Iwa81} (see also Neukirch--Schmidt--Wingberg \cite[Corollary (11.4.11)]{NSW08}).

\begin{thm}\label{thm:main_real}
We have $\mu(X_p(K_{\infty})) = 0$ if and only if $\mu(X_p(k_{\infty})) = 0$.
If these equivalent conditions hold, we have
\[
\lambda(X_p(K_{\infty})) - 1
= [K_{\infty}: k_{\infty}] \left( \lambda(X_p(k_{\infty})) - 1 \right) 
+ \sum_{w \nmid p} (e_w(K_{\infty}/k_{\infty}) - 1),
\]
where $w$ runs over the non-$p$-adic finite primes of $K_{\infty}$.
\end{thm}

\begin{proof}
We apply Theorem \ref{thm:main_ord} to $T = \Z_p(1)$, so $A = \Q_p/\Z_p$ is a trivial representation.
Define the data $\cT^+$ by
\[
T_v^+ = 
\begin{cases}
	0 & \text{if $v \in S_p(k)$,}\\
	T & \text{otherwise.}
\end{cases}
\]
This satisfies Assumption \ref{ass:Sel_rk}.
By the definition of the Selmer group, we have
\[
\Sel_{\cT^+}(A/K_{\infty})^{\vee} \simeq X_p(K_{\infty}).
\]
It is easy to see that Theorem \ref{thm:main_ord} specializes to the theorem.
\end{proof}

\begin{rem}\label{rem:tot_p_ram}
\begin{itemize}
\item[(1)]
In this proof, we applied Theorem \ref{thm:Kida_alg} to the complex $C$ that fits in a triangle
\[
C
\to \RG_{\Iw}(k_S/K_{\infty}, \Z_p(1))
\to \bigoplus_{v \in S_p(k)} \RG_{\Iw}(K_{\infty, v}, \Z_p(1)),
\]
where $S \supset S_{\infty}(k) \cup S_p(k) \cup S_{\ram}(K_{\infty}/k)$.
This $C$ coincides with $C_S[-1]$ in \cite{GKK22}, which was also used to study the Iwasawa module $X_p(K_{\infty})$.

\item[(2)]
We may change the data $\cT^+$ so that
\[
T_v^+ = 
\begin{cases}
	0 & \text{if $v \in S_f$,}\\
	T & \text{otherwise.}
\end{cases}
\]
with $S$ as above.
In this case, the associated complex $C'$ is by definition the compact support cohomology complex.
Concretely, we have $C' = \RG(k_S/K_{\infty}, \Q_p/\Z_p)^{\vee}[-3]$ by the global duality, so its cohomology groups are $H^2(C') \simeq X_S(K_{\infty})$, $H^3(C') \simeq \Z_p$, and $H^i(C') = 0$ for $i \neq 2, 3$.
Then Theorem \ref{thm:main_ord} (or Theorem \ref{thm:Kida_alg}) implies $\mu(X_S(K_{\infty})) = 0 \Leftrightarrow \mu(X_S(k_{\infty})) = 0$ and, assuming these $\mu = 0$, we have
\[
\lambda(X_S(K_{\infty})) - 1 = [K_{\infty}: k_{\infty}] (\lambda(X_S(k_{\infty})) - 1).
\]
We can also deduce Theorem \ref{thm:main_real} from this.
It is enough to compare the $\lambda$-invariants of $X_S$ and $X_p$ by using the short exact sequence
\[
0 \to \bigoplus_{v \nmid p} \cU_{K_{\infty}, v} \to X_S(K_{\infty}) \to X_p(K_{\infty}) \to 0,
\]
where $\cU_{K_{\infty}, v}$ is the projective limit of the $p$-components of the local unit groups.
It is easy to describe its $\Z_p$-ranks, which results in the same formula as Theorem \ref{thm:main_real}.
\item[(3)]
The neat formula in (2) relating $\lambda(X_S(K_{\infty}))$ and $\lambda(X_S(k_{\infty}))$ is not surprising.
Indeed, taking the Kummer duality in mind, we see from \cite[page 276, Theorem 3 and the succeeding text]{Iwa81} that $\mu = 0$ holds for $X_p(k_{\infty})$ if and only if $\Gal(k_S/k_{\infty})$ is a free pro-$p$-group.
Then the formula is a consequence of the Schreier index formula on free groups (or one may apply \cite[page 250, Lemma 7]{Iwa81}).
\end{itemize}
\end{rem}

\subsection{Minus components for CM-fields}\label{ss:Iw2}

Let $p$ be an odd prime number.
Let $K_{\infty}/k_{\infty}$ be a finite $p$-extension between the cyclotomic $\Z_p$-extensions of CM-fields.
As usual, let $K_{\infty}^+$ and $k_{\infty}^+$ denote the maximal totally real subfields of $K_{\infty}$ and $k_{\infty}$, respectively.
Put 
\[
\delta = 
\begin{cases}
	1 & \text{if $\mu_p \subset k_{\infty}$,}\\
	0 & \text{otherwise.}
\end{cases}
\]

Recall that $X(K_{\infty})$ denotes the unramified Iwasawa module for $K_{\infty}$.
As a variant, let $X'(K_{\infty})$ denote the split Iwasawa module for $K_{\infty}$, that is, all finite primes are required to split completely.

The original formula of Kida deals with $X(K_{\infty})^-$.
However, it is not $X(K_{\infty})^-$ that Theorem \ref{thm:main_ord} directly applies to, but $X'(K_{\infty})^-$.
The result is the following.
It is already shown in \cite[Corollary (11.4.13)]{NSW08}, but the methods are different (see the end of this subsection).

\begin{thm}\label{thm:main_CM}
We have $\mu(X'(K_{\infty})^-) = 0$ if and only if $\mu(X'(k_{\infty})^-) = 0$.
If these equivalent conditions hold, then we have
\[
\lambda(X'(K_{\infty})^-) - \delta
= [K_{\infty}: k_{\infty}] (\lambda(X'(k_{\infty})^-) - \delta) + \sum_{w^+} (n_{w^+}(K_{\infty}^+/k_{\infty}^+) - 1),
\]
where $w^+$ runs over the finite primes of $K_{\infty}^+$ that split in $K_{\infty}/K_{\infty}^+$.
\end{thm}

\begin{proof}
We take CM-fields $K/k$ whose cyclotomic $\Z_p$-extensions are $K_{\infty}/k_{\infty}$ such that $K \cap k_{\infty} = k$.
Let $\chi$ be the quadratic character of $\Gal(k/k^+)$.
We consider the representation $\cO_{\chi} = \Z_p$ on which $\Gal(\ol{k^+}/k^+)$ acts as $\chi$.
We apply Theorem \ref{thm:main_ord} to the base field $k^+$, the representation $T = \cO_{\chi}(1)$, and the extension $K_{\infty}^+/k^+$.
The data $\cT^+$ is defined by $T_v^+ = T$ for all finite primes $v$.
Then Assumption \ref{ass:Sel_rk} holds and we have $\Sel_{\cT^+}(A/K_{\infty}^+)^{\vee} \simeq X'(K_{\infty})^-$.
We have
\[
\delta_{\cT^+}(T/K_{\infty}) 
= \lambda(H^0(K_{\infty}^+, \cO_{\chi}(1))) 
= \delta.
\]
It is straightforward to check that the formula in Theorem \ref{thm:main_ord} specializes to the theorem.
\end{proof}

Now let us discuss the original formula of Kida.

\begin{thm}[Kida \cite{Kid80}]\label{thm:Kida}
We have $\mu(X(K_{\infty})^-) = 0$ if and only if $\mu(X(k_{\infty})^-) = 0$.
If these equivalent conditions hold, then we have
\[
 \lambda(X(K_{\infty})^-) - \delta
 = [K_{\infty}: k_{\infty}] (\lambda(X(k_{\infty})^-) - \delta) + \sum_{w^+ \nmid p} (e_{w^+}(K^+_{\infty}/k^+_{\infty}) - 1),
\]
where $w^+$ runs over the non-$p$-adic finite primes of $K_{\infty}^+$ that split in $K_{\infty}/K_{\infty}^+$.
\end{thm}

\begin{proof}
We have two methods to obtain this theorem.

One is to use Theorem \ref{thm:main_real} and the Kummer duality (see \cite[Propositions 13.32 and 15.37]{Was97}) between the $p$-ramified Iwasawa modules for totally real fields and the minus components of the unramified Iwasawa modules for CM-fields.
Strictly speaking, we need a slightly strengthened statement of Theorem \ref{thm:main_real}, but we omit the details.

The other is to use Theorem \ref{thm:main_CM} and the comparison between $X'(K_{\infty})^-$ and $X(K_{\infty})^-$.
For this, we only need to use \cite[Lemma (11.4.9)]{NSW08}, which gives us an exact sequence
\[
0 \to \left(\bigoplus_{w \mid p} \Z_p \right)^- \to X(K_{\infty})^- \to X'(K_{\infty})^- \to 0,
\]
where $w$ runs over the $p$-adic primes of $K_{\infty}$.
\end{proof}

Indeed, these two ideas imply that Theorems \ref{thm:main_real}, \ref{thm:main_CM}, and \ref{thm:Kida} are essentially all equivalent to each other.
Here we mention that the proof of Theorem \ref{thm:main_CM} in \cite{NSW08} relies on these equivalences:
It proceeds as Theorem \ref{thm:main_real} $\Rightarrow$ Theorem \ref{thm:Kida} $\Rightarrow$ Theorem \ref{thm:main_CM}.
On the other hand, our proof of Theorem \ref{thm:main_CM} is more direct.

\subsection{Ordinary CM-fields}\label{ss:Iw4}

Let $k$ be a CM-field such that every $p$-adic prime splits in $k/k^+$.
Take a CM-type $\Sigma \subset S_p(k)$, so $S_p(k)$ is the disjoint union of $\Sigma$ and the conjugate of $\Sigma$.
Let $k_{\infty}/k$ be a fixed $\Z_p$-extension such that no finite prime splits completely (Assumption \ref{ass:spl_fin}).
Let $K_{\infty}/k_{\infty}$ be a finite $p$-extension.
We consider the $\Sigma$-ramified, split Iwasawa module $X'_{\Sigma}(K_{\infty})$.
More concretely, $X'_{\Sigma}(K_{\infty})$ is defined as the Galois group of the maximal abelian pro-$p$ extension over $K_{\infty}$ in which all primes that do not lie above primes in $\Sigma$ split completely.

\begin{thm}\label{thm:main_ord_CM}
We have $\mu(X'_{\Sigma}(K_{\infty})) = 0$ if and only if $\mu(X'_{\Sigma}(k_{\infty})) = 0$.
If these equivalent conditions hold, then we have
\[
\lambda(X'_{\Sigma}(K_{\infty})) - 1
= [K_{\infty}: k_{\infty}] \left( \lambda(X'_{\Sigma}(k_{\infty})) - 1 \right) 
+ \sum_{w \not \in \Sigma(K_{\infty})} (n_w(K_{\infty}/k_{\infty}) - 1),
\]
where $w$ runs over the finite primes of $K_{\infty}$ that do not lie above primes in $\Sigma$.
\end{thm}

\begin{proof}
We apply Theorem \ref{thm:main_ord} to $T = \Z_p(1)$, so $A = \Q_p/\Z_p$ is a trivial representation.
Define the data $\cT^+$ by
\[
T_v^+ = 
\begin{cases}
	0 & \text{if $v \in \Sigma$,}\\
	T & \text{otherwise.}
\end{cases}
\]
This satisfies Assumption \ref{ass:Sel_rk} and we have $\Sel_{\cT^+}(A/K_{\infty})^{\vee} \simeq X'_{\Sigma}(K_{\infty})$.
Now Theorem \ref{thm:main_ord} shows the theorem.
\end{proof}

When $k$ is an imaginary quadratic field and $k_{\infty}/k$ is the unique $\Z_p$-extension unramified outside $\Sigma$, Theorem \ref{thm:main_ord_CM} recovers a recent result of Murakami \cite{Mur}.

\subsection{Ordinary elliptic curves}\label{ss:ell}

Let $E$ be an elliptic curve over a number field $k$.
We suppose that $E$ has good ordinary reduction at all $p$-adic primes of $k$.
For simplicity, we deal with the case where $k_{\infty}/k$ is the cyclotomic $\Z_p$-extension.

The usual Selmer group $\Sel(E/K_{\infty})$ is defined as the kernel of the natural homomorphism
\[
H^1(K_{\infty}, E[p^{\infty}]) \to \bigoplus_{v} \frac{H^1(K_{\infty, v}, E[p^{\infty}])}{E(K_{\infty, v}) \otimes \Q_p/\Z_p}.
\]
Here, $v$ runs over the finite primes of $k$ and $E(K_{\infty, v}) \otimes \Q_p/\Z_p$ is regarded as a subgroup of $H^1(K_{\infty, v}, E[p^{\infty}])$ via the Kummer map.

The following is the result (see also Remark \ref{rem:HM_diff})

\begin{thm}[{Hachimori--Matsuno \cite[Theorem 3.1]{HM99}}]\label{thm:main_ell}
We have $\mu(\Sel(E/K_{\infty})^{\vee}) = 0$ if and only if $\mu(\Sel(E/k_{\infty})^{\vee}) = 0$.
If these equivalent conditions hold, then we have
\begin{align}
\lambda(\Sel(E/K_{\infty})^{\vee})
& = [K_{\infty}: k_{\infty}] \lambda(\Sel(E/k_{\infty})^{\vee})\\
&  \quad + \sum_{w \nmid p, \text{good}} 2(e_w(K_{\infty}/k_{\infty}) - 1)
\quad + \sum_{w \nmid p, \text{split}} (e_w(K_{\infty}/k_{\infty}) - 1)\\
&  \quad + \sum_{w \nmid p, \text{pot-good}} (-2)
\quad + \sum_{w \nmid p, \text{pot-split}} (-1).
\end{align}
Here, $w$ runs over the non-$p$-adic finite primes of $K_{\infty}$:
\begin{itemize}
\item
``good'' means good reduction for $E/k_{\infty}$ and $E(k_v)[p] \neq 0$,
\item
``split'' means split multiplicative reduction for $E/k_{\infty}$,
\item
``pot-good'' means additive reduction for $E/k_{\infty}$ and ``good'' for $E/K_{\infty}$, and
\item
``pot-split'' means additive reduction for $E/k_{\infty}$ and ``split'' for $E/K_{\infty}$.
\end{itemize}
Note that only $w$'s with $e_w(K_{\infty}/k_{\infty}) > 1$ contribute.
\end{thm}

\begin{proof}
We apply Theorem \ref{thm:main_ord} to the Tate module $T = T_pE$, so $A = E[p^{\infty}]$ by the Weil pairing.
For each $v \in S_p(k)$, by the ordinary assumption, we can define $0 \to T_v^+ \to T \to T_v^- \to 0$, or rather $0 \to A_v^+ \to A \to A_v^- \to 0$, as
\[
0 \to \hat{E}_v[p^{\infty}] \to E[p^{\infty}] \to \tilde{E}_v[p^{\infty}] \to 0,
\]
where $\hat{E}_v$ denotes the associated formal group and $\tilde{E}_v$ the reduction modulo $v$.
We also set $T_v^+ = T$ for any $v \nmid p$.
This data $\cT^+$ satisfies Assumption \ref{ass:Sel_rk}.
Then we have $\Sel(E/K_{\infty}) = \Sel_{\cT^+}(A/K_{\infty})$ (see Greenberg \cite[Proposition 2.4]{Gree99}).

By applying Theorem \ref{thm:main_ord}, we obtain the equivalence on $\mu = 0$ and moreover
\begin{align}
& \lambda(\Sel(E/K_{\infty})^{\vee})
	- \lambda(E(K_{\infty})[p^{\infty}]^{\vee})
	- \delta_{\cT^+}(T/K_{\infty})\\
& = [K_{\infty}: k_{\infty}] \left[\lambda(\Sel(E/k_{\infty})^{\vee})
	- \lambda(E(k_{\infty})[p^{\infty}]^{\vee})
	- \delta_{\cT^+}(T/k_{\infty}) \right]\\
& \quad + \sum_{w \mid p} \left[ n_w(K_{\infty}/k_{\infty}) \lambda(\tilde{E}_v(\wtil{k_{\infty, w}})[p^{\infty}]^{\vee}) - \lambda(\tilde{E}_v(\wtil{K_{\infty, w}})[p^{\infty}]^{\vee}) \right]\\
& \quad + \sum_{w \nmid p} \left[ e_w(K_{\infty}/k_{\infty}) \lambda(E(k_{\infty, w})[p^{\infty}]^{\vee}) - \lambda(E(K_{\infty, w})[p^{\infty}]^{\vee}) \right].
\end{align}
Here, $\wtil{(-)}$ denotes the residue fields.
We have to compute the rational points.
\begin{itemize}
\item
(global)
The group $E(K_{\infty})[p^{\infty}]$ is finite, thanks to Mazur \cite[Proposition 6.12]{Maz72} (see \cite[page 57]{Gree99}).
This also implies $\delta_{\cT^+}(T/K_{\infty}) = 0$.
\item
($p$-adic)
For $w \mid p$, the group $\tilde{E}_v(\wtil{K_{\infty, w}})[p^{\infty}]$ is finite, simply because $\wtil{K_{\infty, w}}$ is a finite field.
\item
(non-$p$-adic)
The contribution of non-$p$-adic primes is determined by Proposition \ref{prop:HM1} below.
Note that we only have to consider the primes $w \nmid p$ that are ramified in $K_{\infty}/k_{\infty}$, which implies that $k_v \supset \mu_p$ for the prime $v$ of $k$ lying below $w$.
\end{itemize}
Combining these, we obtain Theorem \ref{thm:main_ell}.
\end{proof}

\begin{prop}[{\cite[Proposition 5.1]{HM99}}]\label{prop:HM1}
Let $F/\Q_l$ ($l \neq p$) be a finite extension with $F \supset \mu_p$ and let $F_{\infty}$ be its cyclotomic $\Z_p$-extension.
Let $E/F$ be an elliptic curve.
Then
\[
\rank_{\Z_p}(E(F_{\infty})[p^{\infty}]^{\vee})
= \begin{cases}
	2 & \text{if $E$ has good reduction and $E(F)[p] \neq 0$,}\\
	1 & \text{if $E$ has split multiplicative reduction,}\\
	0 & \text{otherwise.}
\end{cases}
\]
\end{prop}

\begin{rem}\label{rem:HM_diff}
In \cite{HM99}, the prime $p = 2$ is also dealt with when $k$ is totally imaginary.
In principle our method can be applied to this case, but we omit it.
On the other hand, we have removed an assumption when $p = 3$ (denoted by (Hyp2) in \cite{HM99}) that, for each non-$p$-adic prime, if we have additive reduction for $E/k_{\infty}$, then we still have additive reduction for $E/K_{\infty}$.
This is why the ``pot-good'' or ``pot-split'' term does not appear in \cite[Theorem 3.1]{HM99}.
\end{rem}

\subsection{Supersingular elliptic curves}\label{ss:ell_ss}

As a final application, let us sketch how to deduce Kida's formula for supersingular elliptic curves.
This demonstrates how Theorem \ref{thm:Kida_alg} can also be applied to non-ordinary representations.
Our result is basically the same as \cite[Proposition 5.2]{Lim22} (with minor differences such as Remark \ref{rem:HM_diff}).
Also, when the base field $F$ is $\Q$, the result should coincide with that of \cite[Theorem 4.3]{PW06} specialized to the case the modular form is associated to an elliptic curve, but our formulation is more explicit.
When $F \neq \Q$, we need the author's work \cite {Kata_15} on the precise structures of $\pm$-local conditions.

First we briefly introduce the (multiply-)signed Selmer groups.
The concept of the signed Selmer groups is initiated by Kobayashi \cite{Kob03}.
We follow the notation in \cite{Kata_15}.

Let $E$ be an elliptic curve over a number field $F$.
We suppose that $E$ has good supersingular reduction at all $p$-adic primes of $F$ (just for simplicity, we do not deal with the case where ordinary and supersingular reductions are mixed).
To employ the $\pm$-theory, we moreover assume the following.

\begin{ass}[{\cite[Assumption 2.1]{Kata_15}}]\label{ass:ss}
\begin{itemize}
\item[(1)]
The prime $p$ splits completely in $F/\Q$.
\item[(2)]
For any $p$-adic prime $\fp$ of $F$, we have $a_{\fp}(E) = 0$.
\end{itemize}
\end{ass}

Here, as usual, we set $a_{\fp}(E) = (1 + \# \F_{\fp}) - \# \wtil{E}_{\fp}(\F_{\fp})$, where $\F_{\fp}$ denotes the residue field of $k$ at $\fp$ (which is a prime field by (1)) and $\wtil{E}_{\fp}$ denotes the reduction of $E$ at $\fp$.
Condition (2) automatically holds if $p \geq 5$ by the Hasse bound.

Set $F_{\infty} = F(\mu_{p^{\infty}})$.
Let $K_{\infty}/F_{\infty}$ be a finite extension.
We assume that $K_{\infty}/F$ is abelian in order to use results in \cite{Kata_15}.
Then as in \cite[Definition 2.3]{Kata_15}, for each prime $\fp \in S_p(F)$ and for each sign $\pm$, we have a submodule $E^{\pm}(K_{\infty, \fp})$ of $E(K_{\infty, \fp})$.

Let us take a multi-sign $\epsilon = (\epsilon_{\fp})_{\fp} \in \prod_{\fp \in S_p(F)} \{+, -\}$.
As in \cite[Definition 2.5]{Kata_15}, we define the $\epsilon$-Selmer group $\Sel^{\epsilon}(E/K_{\infty})$ as the kernel of the natural map
\[
H^1(K_{\infty}, E[p^{\infty}])
\to
\bigoplus_{\fp \in S_p(F)} \frac{H^1(K_{\infty, \fp}, E[p^{\infty}])}{E^{\epsilon_{\fp}}(K_{\infty, \fp}) \otimes \Q_p/\Z_p} 
\oplus
\bigoplus_{v \nmid p} H^1(K_{\infty, v}, E[p^{\infty}]),
\]
where $v$ runs over the non-$p$-adic finite primes of $F$.

Now let $k_{\infty}$ be an intermediate field of $K_{\infty}/F_{\infty}$ such that $K_{\infty}/k_{\infty}$ is a $p$-extension.
We also have the $\epsilon$-Selmer group $\Sel^{\epsilon}(E/k_{\infty})$.

\begin{thm}\label{thm:main_ell_ss}
In the situation described above, the statement of Theorem \ref{thm:main_ell} holds after replacing $\Sel$ by $\Sel^{\epsilon}$.
\end{thm}

\begin{proof}
This can be shown by modifying the proof of Theorem \ref{thm:main_ord}.
We consider the representation $T = T_pE$.
The main trouble is that we do not have $T_{\fp}^+$ for $\fp \mid p$.

We remedy this as follows.
For each $\fp \in S_p(F)$, we have a tautological exact sequence
\[
0 \to \left(\frac{H^1(K_{\infty, \fp}, E[p^{\infty}])}{E^{\epsilon_{\fp}}(K_{\infty, \fp}) \otimes \Q_p/\Z_p}  \right)^{\vee}
\to H^1(K_{\infty, \fp}, E[p^{\infty}])^{\vee}
\to (E^{\epsilon_{\fp}}(K_{\infty, \fp}) \otimes \Q_p/\Z_p)^{\vee}
\to 0.
\]
On the other hand, we have 
\[
\RG_{\Iw}(K_{\infty, \fp}, T)
\simeq
H^1_{\Iw}(K_{\infty, \fp}, T)[-1]
\]
by Assumption \ref{ass:ss} and
\[
H^1_{\Iw}(K_{\infty, \fp}, T)
\simeq  H^1(K_{\infty, \fp}, E[p^{\infty}])^{\vee}
\]
by the local duality.
Now let us define
\[
\RG_{\Iw}(K_{\infty, \fp}, T)^+ 
= \left(\frac{H^1(K_{\infty, \fp}, E[p^{\infty}])}{E^{\epsilon_{\fp}}(K_{\infty, \fp}) \otimes \Q_p/\Z_p}  \right)^{\vee}[-1]
\]
and
\[
\RG_{\Iw}(K_{\infty, \fp}, T)^-
= (E^{\epsilon_{\fp}}(K_{\infty, \fp}) \otimes \Q_p/\Z_p)^{\vee}[-1].
\]
We use $\RG_{\Iw}(K_{\infty, \fp}, T)^{\pm}$ as alternatives to $\RG_{\Iw}(K_{\infty, \fp}, T_{\fp}^{\pm})$.

Using this data, we define a complex $C$ in the same way as in \S \ref{ss:Sel_cpx}: we have a triangle
\[
C 
\to \RG_{\Iw}(k_S/K_{\infty}, T)
\to \bigoplus_{\fp \in S_p(F)} \RG_{\Iw}(K_{\infty, \fp}, T_{\fp})^-.
\]
Then properties (i)--(iv) still hold.
To show them, we only have to use \cite[Proposition 4.4]{Kata_15} on the module structure of $(E^{\pm}(K_{\infty, \fp}) \otimes \Q_p/\Z_p)^{\vee}$ (i.e., the finiteness of the projective dimension, the computation of the rank, and the compatibility with respect to the field extension).

Then we can argue in a similar way as in \S \ref{ss:pf_ord_thm}.
We omit the discussion on $\mu = 0$.
We have $H^0(K_{\infty}, E[p^{\infty}]) = 0$ by Assumption \ref{ass:ss}.
Then, assuming $\mu = 0$, we obtain $H^i(C) = 0$ for $i \neq 2$ and an exact sequence
\[
0 
\to \Sel^{\epsilon}(E/K_{\infty})^{\vee} 
\to H^2(C) 
\to \bigoplus_{v \in S_f, v \nmid p} H^0(K_{\infty, v}, E[p^{\infty}])^{\vee}
\to 0,
\]
which are a counterpart of Proposition \ref{prop:C_coh}.
Applying Theorem \ref{thm:Kida_alg}, we obtain a relation between $\lambda(\Sel^{\epsilon}(E/K_{\infty})^{\vee})$ and $\lambda(\Sel^{\epsilon}(E/k_{\infty})^{\vee})$.
The contribution of the local rational points can be determined by Proposition \ref{prop:HM1}.
As a consequence, we obtain the theorem.
\end{proof}

\section{Kida's formula for $p$-adic Lie extensions}\label{sec:Lie}

In this section, we discuss Kida's formula in the context of non-commutative Iwasawa theory.

\subsection{The key algebraic theorem}\label{ss:Kida_alg}

This subsection proceeds in parallel with \S \ref{sec:alg}.
Let $\cG$ be a pro-$p$, $p$-adic Lie group.
Let $\cH$ be a closed normal subgroup of $\cG$ such that $\Gamma := \cG/\cH \simeq \Z_p$.
We assume that $\cG$ has no $p$-torsion or, equivalently, $\cH$ has no $p$-torsion.

The rings $\Z_p[[\cG]]$ and $\Z_p[[\Gamma]]$ play the roles of $\Lambda[G]$ and $\Lambda$ in \S \ref{sec:alg}, respectively.
Note that, since $\cH$ has no $p$-torsion, the situations are disjoint, unless $\cH$ is trivial.
It seems possible to modify the following argument to unify the two cases by allowing $\cH$ has $p$-torsion.
However, we divide them in order to simplify the argument.

We have a notion of $\mu$-invariants for finitely generated $\Z_p[[\cG]]$-modules (e.g., \cite[\S 1.2]{How02}).
If the module is finitely generated over $\Z_p[[\cH]]$, then it satisfies $\mu = 0$.
However, the converse does not hold in general unfortunately, and what is convenient for our purpose is rather being finitely generated over $\Z_p[[\cH]]$.
For this reason, we avoid mentioning the $\mu$-invariants for $\Z_p[[\cG]]$-modules.

Also, instead of $\lambda$-invariants, we explicitly use the notion of the $\Z_p[[\cH]]$-rank, denoted by $\rank_{\Z_p[[\cH]]}(M)$, for a finitely generated $\Z_p[[\cH]]$-module $M$.
The rank is defined as usual by base-changing to the total ring of fractions (see, e.g., \cite[\S 1.1]{How02}).
For any finitely generated $\Z_p[[\cH]]$-module $M$, we have an explicit formula (see \cite[Theorem 1.1]{How02})
\begin{equation}\label{eq:rank_H}
\rank_{\Z_p[[\cH]]}(M) = \sum_{i = 0}^{\infty} (-1)^i \rank_{\Z_p} \left(\Tor_i^{\Z_p[[\cH]]}(\Z_p, M) \right),
\end{equation}
which may be regarded as a definition.
This formula basically results from the observations that the right hand side is equal to the natural rank when $M$ is free over $\Z_p[[\cH]]$ and that the right hand side satisfies the additivity property for exact sequences.
Note also that the right hand side is well-defined because $\Z_p[[\cH]]$ is noetherian and the global dimension of $\Z_p[[\cH]]$ is finite (in fact, is equal to $\dim \cH + 1$) since $\cH$ is $p$-torsion-free.

It is convenient to introduce the {\it canonical Ore set} \cite[\S 2]{CFKSV05}, which is denoted by $\fS$ in the following.
By definition $\fS$ is the set of elements $f \in \Z_p[[\cG]]$ such that the quotient module $\Z_p[[\cG]]/\Z_p[[\cG]]f$ is finitely generated over $\Z_p[[\cH]]$.
It is known that, for a finitely generated $\Z_p[[\cG]]$-module, being $\fS$-torsion is equivalent to being finitely generated over $\Z_p[[\cH]]$ (see \cite[Proposition 2.3]{CFKSV05}).
The set $\fS$ is indeed a multiplicatively closed set satisfying the (both left and right) Ore condition (see \cite[Theorem 2.4]{CFKSV05}).

Note that, using the perspective of \cite{CFKSV05}, we can interpret $\rank_{\Z_p[[\cH]]}(M)$ for $M$ that is finitely generated over $\Z_p[[\cH]]$ by using the {\it Akashi series}.
Indeed, \eqref{eq:rank_H} implies that $\rank_{\Z_p[[\cH]]}(M)$ coincides with the $\lambda$-invariant of the Akashi series of $M$ defined in \cite[\S 4, Equation (40)]{CSS03} or \cite[\S 3, Equation (37)]{CFKSV05}.

As an analogue of Proposition \ref{prop:alg1}, we observe the following.

\begin{prop}\label{prop:alg1A}
Let $M$ be a finitely generated $\Z_p[[\cG]]$-module.
Set $\ol{M} = \Z_p[[\Gamma]] \otimes_{\Z_p[[\cG]]} M$.
\begin{itemize}
\item[(1)]
The module $M$ is finitely generated over $\Z_p[[\cH]]$ if and only if $\mu(\ol{M}) = 0$.
\item[(2)]
Suppose $\pd_{\Z_p[[\cG]]}(M) \leq 1$.
If the equivalent conditions in (1) hold, then we have $\rank_{\Z_p[[\cH]]}(M) = \lambda(\ol{M})$.
\end{itemize}
\end{prop}

\begin{proof}
(1)
This follows from topological Nakayama's lemma as in Proposition \ref{prop:alg1}(1).

(2) 
We take an exact sequence $0 \to F' \to F \to M \to 0$ with finitely generated free $\Z_p[[\cG]]$-modules $F$, $F'$.
Since $M$ is $\fS$-torsion, by localizing with respect to $\fS$, we see that the ranks of $F$ and $F'$ are the same.
By base-changing to $\Z_p[[\Gamma]]$, we obtain an exact sequence $0 \to \ol{F'} \to \ol{F} \to \ol{M} \to 0$.
Here, the injectivity of $\ol{F'} \to \ol{F}$ holds since the ranks over $\Z_p[[\Gamma]]$ are the same and $\mu(\ol{M}) = 0$.
Thus, we have $\Tor_i^{\Z_p[[\cH]]}(\Z_p, M) = 0$ for $i \geq 1$, so \eqref{eq:rank_H} implies
\[
\rank_{\Z_p[[\cH]]}(M) = \rank_{\Z_p} (\ol{M}) = \lambda(\ol{M}).
\]
This completes the proof.
\end{proof}

Let us prepare a lemma for the next theorem.

\begin{lem}\label{lem:exists_morph}
Let $\varphi: F' \to F$ be a homomorphism between $\Z_p[[\cG]]$-modules.
Suppose that $F$ is finitely generated and free over $\Z_p[[\cG]]$ and that the cokernel of $\varphi$ is finitely generated over $\Z_p[[\cH]]$.
Then we can construct a commutative diagram
\[
\xymatrix{
	F \ar[r]^{\psi} \ar[d]_{*}
	& F \ar[d]^{=}\\
	F' \ar[r]_{\varphi}
	& F
}
\]
such that the cokernel of $\psi$ is finitely generated over $\Z_p[[\cH]]$.
\end{lem}

\begin{proof}
Let $x_1, \dots, x_n$ be a basis of $F$ as a $\Z_p[[\cG]]$-module.
For each $1 \leq i \leq n$, by assumption there is an element $f_i \in \fS$ such that $f_i x_i$ is in the image of $\varphi$.
We define the map $\psi$ by sending $x_i$ to $f_i x_i$.
Then the map $*$ that makes the diagram commutative exists by the choice of $f_i$.
Also, the cokernel of $\psi$ is isomorphic to $\bigoplus_{i = 1}^n \Z_p[[\cG]]/\Z_p[[\cG]]f_i$, which is finitely generated over $\Z_p[[\cH]]$ by $f_i \in \fS$.
\end{proof}

Now we move on to complexes.
As in \S \ref{ss:alg_cpx}, for a perfect complex $C$ over $\Z_p[[\cG]]$ whose cohomology groups are all finitely generated over $\Z_p[[\cH]]$, we define
\[
\rank_{\Z_p[[\cH]]}(C) = \sum_{i \in \Z} (-1)^i \rank_{\Z_p[[\cH]]}(H^i(C)).
\]

The following is the analogue of Theorem \ref{thm:Kida_alg}.

\begin{thm}\label{thm:Kida_algA}
Let $C$ be a perfect complex over $\Z_p[[\cG]]$ and set $\ol{C} = \Z_p[[\Gamma]] \otimesL_{\Z_p[[\cG]]} C$.
Then the cohomology groups of $C$ are all finitely generated over $\Z_p[[\cH]]$ if and only if $\mu = 0$ holds for $\ol{C}$.
If these equivalent conditions hold, then we have $\rank_{\Z_p[[\cH]]}(C) = \lambda(\ol{C})$.
\end{thm}

\begin{proof}
We can prove this theorem exactly in the same way as Theorem \ref{thm:Kida_alg}.
We apply Proposition \ref{prop:alg1A} instead of Proposition \ref{prop:alg1}.
In order to proceed the induction argument, we have to construct a complex $C'$ and a morphism from $C'$ to $C$, which is possible by Lemma \ref{lem:exists_morph}.
\end{proof}

\begin{rem}
In the situation of Theorem \ref{thm:Kida_algA}, we see that $C$ is a perfect complex over $\Z_p[[\cH]]$ and $\rank_{\Z_p[[\cH]]}(C)$ can be regarded as the definition of the Euler characteristic of $C$ over $\Z_p[[\cH]]$.
Then Theorem \ref{thm:Kida_algA} claims simply that the Euler characteristic remains unchanged after base-change.
This interpretation is also valid for Theorem \ref{thm:Kida_alg}.
\end{rem}

\subsection{Arithmetic applications}\label{ss:Kida_inf}

As in \S \ref{ss:state_thm}, we consider an odd prime number $p$, a number field $k$, and a Galois representation $T$ with ordinarity data $\cT^+$ satisfying Assumption \ref{ass:Sel_rk}.
We also assume to be given a $\Z_p$-extension $k_{\infty}/k$ satisfying Assumption \ref{ass:spl_fin}.

Let $K_{\infty}/k$ be a pro-$p$, $p$-adic Lie extension satisfying $K_{\infty} \supset k_{\infty}$.
Set $\cG = \Gal(K_{\infty}/k)$ and $\cH = \Gal(K_{\infty}/k_{\infty})$, so $\Gamma = \cG/\cH = \Gal(k_{\infty}/k)$.
We assume that $\cG$ has no $p$-torsion and $\dim \cG \geq 2$.
We finally assume that $K_{\infty}/k_{\infty}$ is unramified at almost all finite primes.

We then consider the Selmer group $\Sel_{\cT^+}(A/K_{\infty})$ defined in the same way as before.
The following is the analogue of Theorem \ref{thm:main_ord}.

\begin{thm}\label{thm:main_ordA}
The module $\Sel_{\cT^+}(A/K_{\infty})^{\vee}$ is finitely generated over $\Z_p[[\cH]]$ if and only if $\mu(\Sel_{\cT^+}(A/k_{\infty})^{\vee}) = 0$.
If these equivalent conditions hold, then we have
\begin{align}
\rank_{\Z_p[[\cH]]}(\Sel_{\cT^+}(A/K_{\infty})^{\vee})
& = \lambda(\Sel_{\cT^+}(A/k_{\infty})^{\vee})
	- \lambda(H^0(k_{\infty}, A)^{\vee})
	- \delta_{\cT^+}(T/k_{\infty})\\
& \qquad \quad + \sum_{w} \lambda(H^0(k_{\infty, w}, A_v^-)^{\vee}),
\end{align}
where $w$ runs over the finite primes of $k_{\infty}$ (not $K_{\infty}$) that does not split completely in $K_{\infty}/k_{\infty}$.
\end{thm}

\begin{proof}
This theorem is proved in the same way as Theorem \ref{thm:main_ord}.
We only discuss how to obtain the formula on the ranks.
Assuming $\mu = 0$, a direct application of Theorem \ref{thm:Kida_algA} to the appropriate Selmer complex implies
\begin{align}
& \rank_{\Z_p[[\cH]]}(\Sel_{\cT^+}(A/K_{\infty})^{\vee})
	- \rank_{\Z_p[[\cH]]}(H^0(K_{\infty}, A)^{\vee})
	- \delta_{\cT^+}(T/K_{\infty})\\
& = \lambda(\Sel_{\cT^+}(A/k_{\infty})^{\vee})
	- \lambda(H^0(k_{\infty}, A)^{\vee})
	- \delta_{\cT^+}(T/k_{\infty})\\
& \quad + \sum_{v \in S_f} \left[ \lambda(H^0(k_{\infty, v}, A_v^-)^{\vee}) - \rank_{\Z_p[[\cH]]}(H^0(K_{\infty, v}, A_v^-)^{\vee}) \right]
\end{align}
(the meaning of $\delta_{\cT^+}(T/K_{\infty})$ is clear) with $S$ sufficiently large.
Because of $\dim \cH \geq 1$, both $\rank_{\Z_p[[\cH]]}(H^0(K_{\infty}, A)^{\vee})$ and $\delta_{\cT^+}(T/K_{\infty})$ vanish.
For a finite prime $w$ of $k_{\infty}$, if $w$ splits completely in $K_{\infty}/k_{\infty}$, then we have $\lambda(H^0(k_{\infty, w}, A_v^-)^{\vee}) = \rank_{\Z_p[[\cH]]}(H^0(K_{\infty, w}, A_v^-)^{\vee})$.
Otherwise, since the decomposition subgroup in $\cH$ is infinite, we have $\rank_{\Z_p[[\cH]]}(H^0(K_{\infty, w}, A_v^-)^{\vee}) = 0$.
Thus, we obtain the desired formula.
\end{proof}

Note that, by Assumption \ref{ass:spl_fin}, a non-$p$-adic finite prime $w$ of $k_{\infty}$ splits completely in $K_{\infty}/k_{\infty}$ if and only if it is unramified in $K_{\infty}/k_{\infty}$.

We are able to apply Theorem \ref{thm:main_ordA} to various arithmetic situations as in \S \S \ref{ss:Iw}--\ref{ss:ell}.
Here we only illustrate the results in the situation in \S \ref{ss:Iw2} and omit the others just because we encounter no additional difficulties.

\begin{eg}
We consider the situation in \S \ref{ss:Iw2}, but we assume that $K_{\infty}/k_{\infty}$ is an infinite extension such that $\cH = \Gal(K_{\infty}/k_{\infty})$ is a non-trivial, pro-$p$, $p$-adic Lie group without $p$-torsion.
Then as a variant of Theorem \ref{thm:main_CM}, we obtain the equivalence for $\mu = 0$ and a formula
\[
\rank_{\Z_p[[\cH]]}(X'(K_{\infty})^-) 
= \lambda(X'(k_{\infty})^-) - \delta + \# Q'_{K_{\infty}/k_{\infty}},
\]
where $Q'_{K_{\infty}/k_{\infty}}$ denotes the set of finite primes of $k_{\infty}^+$ that does not split completely in $K_{\infty}^+/k_{\infty}^+$ and splits in $k_{\infty}/k_{\infty}^+$.

As in Theorem \ref{thm:Kida}, from this we can also deduce a formula
\[
\rank_{\Z_p[[\cH]]}(X(K_{\infty})^-) 
= \lambda(X(k_{\infty})^-) - \delta + \# Q_{K_{\infty}/k_{\infty}},
\]
where $Q_{K_{\infty}/k_{\infty}}$ denotes the set of non-$p$-adic primes in $Q'_{K_{\infty}/k_{\infty}}$.
This formula is nothing but a theorem of Hachimori--Sharifi \cite[Theorem 1.2]{HS05}.
\end{eg}

\begin{rem}
Bhave \cite[\S 4]{Bha07} obtained another kind of Kida's formula for finite extensions between $p$-adic Lie extensions (i.e., the behavior of the ranks when $K_{\infty}$ in this section varies along finite extensions).
We can obtain such a formula by modifying our method in this paper or, alternatively, by combining Theorems \ref{thm:main_ord} and \ref{thm:main_ordA} themselves.
\end{rem}

{
\bibliographystyle{abbrv}
\bibliography{biblio}
}

\end{document}

%% file: mystyle.tex
\usepackage{framed}
\usepackage{braket}
\usepackage{amsmath,amssymb,amsthm}
\usepackage{color}
\usepackage{bm}
\usepackage[all]{xy}
\usepackage{amscd}
\usepackage{graphicx}
\usepackage{ascmac}
\usepackage{BOONDOX-frak, mathrsfs}
\usepackage[top=30truemm,bottom=30truemm,left=25truemm,right=25truemm]{geometry}
\usepackage{mathtools}
\mathtoolsset{showonlyrefs=true}
\usepackage{tikz}
\usetikzlibrary{cd}

\makeatletter
\@addtoreset{equation}{section}

\makeatother

\makeatletter
\@namedef{subjclassname@2020}{\textup{2020} Mathematics Subject Classification}
\makeatother

\newcommand{\Z}{\mathbb{Z}}

\newcommand{\Q}{\mathbb{Q}}
\newcommand{\R}{\mathbb{R}}
\newcommand{\C}{\mathbb{C}}
\newcommand{\F}{\mathbb{F}}

\newcommand{\bG}{\mathbb{G}}

\newcommand{\fS}{\mathfrak{S}}

\newcommand{\fp}{\mathfrak{p}}

\newcommand{\cG}{\mathcal{G}}
\newcommand{\cH}{\mathcal{H}}

\newcommand{\cO}{\mathcal{O}}

\newcommand{\cR}{\mathcal{R}}

\newcommand{\cT}{\mathcal{T}}
\newcommand{\cU}{\mathcal{U}}

\newcommand{\wtil}[1]{\widetilde{#1}}
\newcommand{\ol}[1]{\overline{#1}}

\DeclareMathOperator{\Gal}{Gal}

\DeclareMathOperator{\Ker}{Ker}

\DeclareMathOperator{\pd}{pd}

\DeclareMathOperator{\ram}{ram}

\DeclareMathOperator{\id}{id}

\DeclareMathOperator{\rank}{rank}

\usepackage[OT2,T1]{fontenc}
\DeclareSymbolFont{cyrletters}{OT2}{wncyr}{m}{n}
\DeclareMathSymbol{\Sha}{\mathalpha}{cyrletters}{"58}

\let\oldenumerate\enumerate
\renewcommand{\enumerate}{
   \oldenumerate
   \setlength{\itemsep}{1pt}
   \setlength{\parskip}{0pt}
   \setlength{\parsep}{0pt}
}
\let\olditemize\itemize
\renewcommand{\itemize}{
   \olditemize
   \setlength{\itemsep}{1pt}
   \setlength{\parskip}{0pt}
   \setlength{\parsep}{0pt}
}

\theoremstyle{plain}
\newtheorem{thm}{Theorem}[section]
\newtheorem{lem}[thm]{Lemma}

\newtheorem{prop}[thm]{Proposition}

\newtheorem{ass}[thm]{Assumption}

\theoremstyle{definition}
\newtheorem{defn}[thm]{Definition}

\newtheorem{rem}[thm]{Remark}
\newtheorem{eg}[thm]{Example}